\documentclass[a4paper]{amsart}
\usepackage[all]{xy}
\usepackage{amsmath}
\usepackage{amsfonts,amssymb}
\usepackage{hyperref}

\usepackage{booktabs}

\numberwithin{equation}{section}
\newtheorem{satz}{Satz}
\newtheorem{proposition}[satz]{Proposition}
\newtheorem{theorem}{Theorem}
\newtheorem*{theorem*}{Theorem}

\newtheorem{lemma}[satz]{Lemma}

\newtheorem{corollary}[satz]{Corollary}

\theoremstyle{definition}
\newtheorem*{remark*}{Remark}
\newtheorem{remark}[satz]{Remark}

\newtheorem{example}{Example}

\newcommand{\tensor}{\otimes}

\newcommand{\map}[1]{\stackrel{#1}{\longrightarrow}}

\newcommand{\un}[1]{\ensuremath{\protect\underline{#1}}}

\def\GL{\textrm{GL}}

\def\PGL{\textrm{PGL}}

\DeclareMathOperator{\Pic}{Pic}

\DeclareMathOperator{\Bun}{Bun}

\DeclareMathOperator{\Chain}{Chain}
\DeclareMathOperator{\PChain}{PChain}
\DeclareMathOperator{\PPChain}{PPChain}
\DeclareMathOperator{\Higgs}{Higgs}

\DeclareMathOperator{\Sym}{Sym}
\DeclareMathOperator{\Lie}{Lie}

\DeclareMathOperator{\rk}{rk}

\DeclareMathOperator{\Var}{Var}
\DeclareMathOperator{\id}{id}

\DeclareMathOperator{\im}{Im}
\DeclareMathOperator{\coker}{coker}

\DeclareMathOperator{\LHS}{LHS}
\DeclareMathOperator{\sst}{ss}

\def\K0hat{\widehat{K}_0(\Var_k)}
\def\KhatC{\widehat{K}_0(\Var_{\bC})}

\def\1halb{\frac{1}{2}}

\def\sxymat{\xymatrix@C=1.5ex@R=0.8ex}
\def\grp{$\xymatrix{ R\times_{X}R  \ar[r]^-{\mu} & R \ar@<1ex>[r]^-{s}\ar@<-1ex>[r]_-{t} & X}$}
\def\dar{\ar@<-0.5ex>[r]\ar@<0.5ex>[r]}
\def\tar{\ar[r]\ar@<1ex>[r]\ar@<-1ex>[r]}
\newcommand{\dmap}[2]{\ar@<-0.5ex>[r]_-{#2}\ar@<0.5ex>[r]^-{#1}}
\newcommand{\dotarrow}[2]{\xymatrix{{#1}\ar@{..>}[r]&{#2}}}
\def\cart{\ar@{}[dr]|{\square}}

\def\cE{\mathcal{E}}
\def\cF{\mathcal{F}}

\def\cK{\mathcal{K}}

\def\cM{\mathcal{M}}

\def\cO{\mathcal{O}}
\def\cP{\mathcal{P}}

\def\bA{{\mathbb A}}

\def\bC{{\mathbb C}}

\def\bG{{\mathbb G}}

\def\bL{{\mathbb L}}

\def\bN{{\mathbb N}}

\def\bP{{\mathbb P}}
\def\bQ{{\mathbb Q}}
\def\bR{{\mathbb R}}

\def\bZ{{\mathbb Z}}

\begin{document}
\SelectTips{cm}{}

\title[A conjecture of Hausel]{The $y$ genus of the moduli space of $\PGL_n$-Higgs bundles on a curve (for degree coprime to $n$)}
\author[O. Garcia-Prada]{Oscar Garc\'ia-Prada}
	\address{Instituto de Ciencias Matem\'aticas CSIC-UAM-UC3M-UCM,
	calle Nicol\'as Cabrera, 15, Campus de Cantoblanco, 28049 Madrid, Spain}
	\email{oscar.garcia-prada@icmat.es}
\author[J. Heinloth]{Jochen Heinloth}
\address{Universit\"at Duisburg--Essen, Fachbereich Mathematik, Universit\"atsstrasse 2, 45117 Essen, Germany}
\email{Jochen.Heinloth@uni-due.de}

\begin{abstract}
Building on our previous joint work with A. Schmitt \cite{HGS} we explain a recursive algorithm to determine the cohomology of moduli spaces of Higgs bundles on any given curve (in the coprime situation). As an application of the method we compute the $y$-genus of the space of $\PGL_n$-Higgs bundles for any rank $n$, confirming a conjecture of T. Hausel.
\end{abstract}
\maketitle

\section{Introduction}


In the article \cite{HGS} we explained a strategy to compute the cohomology of the moduli space of semistable Higgs bundles of coprime rank and degree on a smooth projective curve. However, we were only able to fully implement this strategy in some cases, because of a convergence problem in one of the steps of our approach. In this article we explain how one can bypass this convergence problem using a wall crossing procedure (Theorem \ref{factorP1} and Remark \ref{algorithm}).

Although the recursive procedure to compute the cohomology of the moduli space that we find is quite involved, T.\ Hausel suggested \cite[Conjecture 5.7]{HauselMirror} that the procedure should simplify for a particular cohomological invariant of the moduli space $\cM^{d,\sst}(\PGL_n)$ of semistable $\PGL_n$ Higgs bundles, called the $y$-genus. For varieties $X$ over $\bC$ the (compactly supported) $y$-genus is defined as 
$$H_y(X,y):=\sum_{p,q\geq 0} (-1)^{k} h^{k;p,q}_c(X) y^p = E(X,1,y),$$ where $h^{k;p,q}_c(X)$ is the weight $p,q$ part of the compactly supported cohomology group $H^k_c(X)$ (see Section \ref{main}). In particular  $H_y(X,1)$ is the Euler characteristic of $X$.

Our main result confirms the conjecture of Hausel on the $y$-genus for arbitrary rank $n$:

\begin{theorem*}[Hausel's conjecture for the $y$-genus]
Let $n\in \bN$ and $d\in \bZ$ be coprime. Then the  $y$-genus of the moduli space of semistable $\PGL_n$-Higgs bundles of degree $d$ on a curve of genus $g$ is given by:
\begin{align*}
H_y(\cM^{d,\sst}(\PGL_n),y)=& y^N \left(\frac{1-y^n}{1-y}\right)^{g-1} \sum_{m|n}  \frac{\mu(m)}{m} \Big(m \prod_{j=1}^{\frac{n}{m}-1}(1-y^{jm})^2   \Big)^{g-1},
\end{align*}
where $N=(n^2-1)(g-1)=\1halb \dim(\cM^{d,\sst}(\PGL_n)$ and $\mu$ denotes the M\"obius $\mu$-function.
\end{theorem*}

Let us recall that Hausel obtained his conjecture as a specialization of the conjecture of Hausel and Rodriguez-Villegas \cite{HRV} which gives an expression for the E-polynomial of the moduli space of semistable Higgs bundles. So far the conjectured formula for the E-polynomial has only been confirmed for $n\leq 4$. 

Let us briefly indicate how we obtain our results and thereby give an overview of the article. In Section \ref{sec:preliminaries} we recall the setup used in \cite{HGS}, expressing cohomological computations in terms of the Grothendieck ring of varieties and reducing the calculation for Higgs bundles to a calculation for moduli spaces of chains.

In Section \ref{sec:wallcrossing} we explain how the classical approach to study the geometry of moduli spaces of chains by variation of the stability parameter can be described from a stack theoretic point of view. In particular, the difference between moduli spaces for different stability parameters admits an inductive description, as long as the stability parameter is not smaller than the one arising from the stability condition for Higgs bundles (condition ($\star$) defined in Section \ref{sec:preliminaries}).

In order to apply this result one needs to find extremal stability parameters, for which the moduli space of semistable objects admits an explicit description.
In Section \ref{sec:conditions} we therefore give a list of necessary conditions for the moduli space of semistable chains to be non empty (Proposition \ref{conditions}). As a corollary we find that for fixed rank and stability parameter the space of semi stable chains can only be nonempty for only finitely many values of the degree of the chains (Corollary \ref{finite}).

In Section \ref{sec:application} we apply these results to obtain a recursive method to compute the class of the moduli space of semistable chains for any stability parameter satisfying condition ($\star$). In particular this shows that the class of these spaces can be expressed purely in terms of the symmetric powers of the curve and that most of the classes are divisible by a power of the Jacobian of the curve (Theorem \ref{factorP1}). This result allows us to compute the compactly supported $y$-genus of the moduli space of $\PGL_n$-Higgs bundles in Section \ref{main}, because only very few spaces of semistable chains can have non trivial $y$-genus.

\noindent{\bf Acknowledgments:} We are greatly indebted to T.\ Hausel and A.\ Schmitt, without discussions with them, this article would not exist. For the results of Section 4, numerical data provided by F.\ Rodriguez-Villegas was essential.  The core of this article was written during the Research Programme on Geometry and Quantization of Moduli Spaces at CRM. We would like to thank the organizers and the CRM for their generous support.  In particular we thank I.\ Mundet i Riera for many stimulating discussions.


\section{Preliminaries: Notation and a reminder of Chains and Higgs bundles}\label{sec:preliminaries}

\subsection{Reminder on the Grothendieck ring of varieties}
As in \cite{HGS} we will do our cohomology calculations by using geometric decompositions of the varieties in question. Thus, we find it convenient
to formulate our computations in the Grothendieck ring of varieties over a field $k$. We will denote this ring by $K_0(\Var_k)$, it is the free abelian group generated by isomorphism classes of quasi-projective varieties $[X]$ subject to the relation $[X] = [X-Z] +[Z]$ for all closed subvarieties $Z \subset X$. The fibre product defines a ring structure on this abelian group. The Lefschetz class is denoted by $\bL:=[\bA^1] \in K_0(\Var_k)$. 

Finally, $\K0hat$ denotes the dimensional completion of $K_0(\Var_k)[\frac{1}{\bL}]$, i.e., the completion with respect to the filtration given by the ideals $I_d$, which are generated by $\bL^{-m}[X]$ where $\dim X - m \leq -d$. 

Throughout this article, $C$ will be a fixed, smooth projective, geometrically irreducible curve, defined over a field $k$. All terms entering in our computation can be expressed in terms of the Zeta function of the curve:
$$ Z(C,t):= \sum_{i\geq 0} [\Sym^i(C)] t^i \in \K0hat[[t]].$$
We recall (\cite[Section 3]{FranziskaZeta}) that this Zeta function is a rational function, namely setting $P(t):=\sum_{i=0}^{2g} \Sym^i([C]-[\bP^1]) t^i$ one has: 
$$Z(C,t) = \frac{P(t)}{(1-t)(1-\bL t)}.$$

In case $k=\bC$ is the field of complex numbers the E-polynomial of varieties defines a ring homomorphism:
$$E\colon K_0(\Var_\bC) \to \bZ[x,y]$$
defined by $$E([X],x,y):= \sum_{k=0}^{2 \dim X} \sum_{0\leq p,q} (-1)^k h_c^{k;p,q}(X) x^p y^q$$ whenever $X$ is a variety over $\bC$.
Since $E(\bL,x,y)=xy$ and $h_c^{k;p,q}(X)=0$ if $k>2\dim X$ this extends to a ring homomorphism
$$E\colon \KhatC \to \bZ[x,y][[(xy)^{-1}]].$$

The compactly supported $y$-genus of a complex variety is $H_y([X],y):=E([X],1,y)$, which we can be viewed as a ring homomorphism:
$$ H_y\colon \KhatC \to \bZ[y][[\frac{1}{y}]] = \bZ((y^{-1})).$$

The basic examples which we will need are:
\begin{example}\label{example}
$E(Z(C,t),x,y)=\frac{(1-xt)^g(1-yt)^g}{(1-t)(1-xyt)},\text{ and }$ $E(P(t))=(1-xt)^g(1-yt)^g.$

In particular we find 
\begin{align*}
H_y(P(t),y)&=(1-t)^g(1-yt)^g,\\
H_y(Z(C,t),y)&=(1-t)^{g-1}(1-yt)^{g-1}.
\end{align*}
\end{example}

\subsection{Moduli spaces of Higgs bundles and Chains}

As before we will denote by $C$ a fixed smooth projective curve of genus $g$ over some field $k$. We will denote by $\Bun_n^d$ the moduli stack of vector bundles of rank $n$ and degree $d$ on $C$. The slope of a vector bundle $\cE$ is denoted by $\mu(\cE):=\frac{\deg(\cE)}{\rk(\cE)}$.

We will denote by $\cM_n^{d}$ the moduli stack of Higgs bundles of rank $n$ and degree $d$ on $C$, i.e., it is the moduli stack classifying pairs  $(\cE,\theta\colon \cE \to \cE \tensor \Omega_C)$, where $\cE\in\Bun_n^d$ is a vector bundle on $C$, $\theta$ is an $\cO_C$-linear map and $\Omega_C$ is the sheaf of differential on $C$. 

Recall that a Higgs bundle $(\cE,\theta)$ is called semistable if for all subsheaves $0\subsetneq \cF\subset \cE$ that are preserved by $\theta$ we have 
$\mu(\cF) \leq \mu(\cE)$. We will denote by $\cM_n^{d,\sst}\subset \cM_n^d$ the substack of semistable Higgs bundles and $M_n^d$ denotes the corresponding coarse moduli space. If $(n,d)=1$, the stack $\cM_n^{d,\sst}$ is a neutral $\bG_m$-gerbe over $M_n^d$. In particular we have 
$$[M_n^d]=(\bL-1)[\cM_n^{d,\sst}]\in \K0hat,$$
so that the class of the moduli space is determined by the class of the moduli stack of semistable Higgs bundles. 
Moreover, again assuming $(n,d)=1$ it is known (e.g. \cite[Theorem 2.1]{HauselMirror}) that the cohomology of $M_n^d$ is pure, so that the E-polynomial of $M_n^d$ determines the dimensions of the cohomology groups of $M_n^d$.

Let us recall that the starting point for our computation in \cite{HGS} was the observation due to Hitchin, Simpson, Hausel and Thaddeus, that the class of $[\cM_n^d]$ can be expressed in terms of moduli stacks of semistable chains. Namely, the multiplicative group acts on $\cM_n^d$, by scaling the homomorphism $\theta$ and the fixed points under this action can be interpreted as moduli stacks of chains.

Let us denote by $\Chain_{\un{n}}^{\un{d}}$ the moduli stack of chains rank $\un{n}\in \bN^{r+1}$ and degree $\un{d}\in\bZ^{r+1}$ on  $C$, i.e., this is the algebraic stack classifying  collections $\cE_\bullet:=((\cE_i)_{i=0,\dots,r},(\phi_i)_{i=1,\dots,r})$, where $\cE_i$ are vector bundles of rank $n_i$ and degree $d_i$ on $C$ and $\phi_i\colon\cE_i\to\cE_{i-1}$ are morphisms of $\cO_C$-modules. 

To any chain $\cE_\bullet$ one can associate a Higgs bundle $(\oplus \cE_i \tensor \Omega^i, \oplus \phi_i)$, which induces a notion of stability for $\cE_\bullet$. However there is  a more flexible notion of stability for chains, which depends on a parameter $\un{\alpha} \in \bR^{r+1}$. 
For a chain $\cE_\bullet=(\cE_i,\phi_i)$ the $\un{\alpha}$-slope is defined to be 
$$ \mu_{\un{\alpha}} (\cE_\bullet) := \frac{\sum_{i=0}^{r} \deg(\cE_i) + \rk(\cE_i) \alpha_i}{\sum_{i=0}^{r} \rk(\cE_i)}.$$

A chain $\cE_\bullet$ is called $\un{\alpha}$-semistable, if for all subchains $0\neq \cF_\bullet \subset \cE_\bullet$ we have 
$$ \mu_{\un{\alpha}}(\cF_\bullet) \leq \mu_{\un{\alpha}}(\cE_\bullet).$$
Again we write $\Chain_{\un{n}}^{\un{d},\un{\alpha}-\sst}\subset \Chain_{\un{n}}^{\un{d}}$ for the substack of $\un{\alpha}$-semistable chains.

Tracing back definitions one finds that the Higgs bundle associated to a chain $\cE_\bullet$ is semistable if and only if the chain is $\un{\alpha}$-semistable for the parameter $\un{\alpha}=(\alpha_i)$ with $\alpha_i=i(2g-2)$. 

We will say that $\alpha$ satisfies ($\star$) if 
\begin{align*}
 \alpha_i-\alpha_{i-1}&\geq 2g-2 \textrm{ for } i=1,\dots r.  \tag{$\star$}
\end{align*}

With this notation we can finally recall how to express the class $[\cM_n^{d,\sst}]$ in terms of moduli spaces of chains.
Write $I:=\{ (\un{n},\un{d}) | \sum_i n_i =n, \sum_i (d_i+i(2g-2)n_i)=d \}$. Then it is known (\cite[Proposition 9.1]{HT},\cite[Section 2]{HGS}) that 
$$[\cM_n^{d,\sst}] =  \bL^{n^2(g-1)} \sum_{(\un{n},\un{d})\in I} [\Chain_{\un{n}}^{\un{d},\un{\alpha}-\sst}]\in \K0hat.$$

Finally let us recall that any $\un{\alpha}$-unstable chain $\cE_\bullet$ admits a unique Harder--Narasim\-han filtration, i.e., a filtration $0=\cF_\bullet^0 \subsetneq \cF_\bullet^1 \subsetneq \dots \subsetneq \cF_\bullet^h \subsetneq \cF_\bullet^{h+1}=\cE_\bullet$ such that all subquotients $\cF_\bullet^i/\cF_\bullet^{i-1}$ are semistable and the slopes of the subquotients are decreasing: $\mu_{\un{\alpha}}(\cF_\bullet^1) > \mu_{\un{\alpha}}(\cF_\bullet^2/\cF_\bullet^1) > \dots > \mu_{\un{\alpha}}(\cF_\bullet^{h+1}/ \cF_\bullet^h).$ 

The collection $(\rk(\cF^i_\bullet/\cF^{i-1}_\bullet),\deg(\cF^i_\bullet/\cF^{i-1}_\bullet))_{i=1,\dots h+1}$ is called the type $\un{t}_\alpha$ of the Harder--Narasimhan filtration. Unstable chains such that the Harder-Narasimhan filtration is of some fixed type define locally closed substacks $\Chain_{\un{n}}^{\un{d},\text{type}-\un{t}_{\alpha}}$ of $\Chain_{\un{n}}^{\un{d}}$ called the Harder--Narasimhan strata.

From \cite[Proposition 4.8]{HGS} we know that for stability parameters $\un{\alpha}$ satisfying ($\star$)
we have that for a type $t_{\un{\alpha}}=(\un{n}^i,\un{d}^i)$: 
$$ [\Chain_{\un{n}}^{\un{d},\text{type}-\un{t}_{\alpha}}] = \bL^{\chi} \prod_{i=1}^{h+1} [\Chain_{\un{n}_i}^{\un{d}_i,\un{\alpha}-\sst}] \in \K0hat,$$
where $\chi\in \bZ$ is given by some explicit formula in terms of the type $\un{t}_\alpha$.

\section{Wall crossing from the point of view of stacks}\label{sec:wallcrossing}

On the moduli stack of chains of rank $\un{n}\in \bN^{r+1}$ and degree $\un{d}\in \bZ^{r+1}$ we have a family of stability conditions 
parameterized by a parameter $\un{\alpha}\in\bR^{r+1}$. Since the slope of a chain only depends on these numerical invariants we will also write
$$\mu_{\un{\alpha}} (\un{n},\un{d}):= \frac{\sum d_i +\alpha_i n_i }{\sum n_i}.$$

The parameter $\un{\alpha}$ is called a {\em critical value} if there exists $\un{n}^\prime,\un{d}^\prime$ with $\un{n}^\prime<\un{n}$ such that 
$\mu_{\un{\alpha}}(\un{n}^\prime,\un{d}^\prime)=0$ and $\mu_{\un{\gamma}}(\un{n}^\prime,\un{d}^\prime)\neq 0$ for some $\un{\gamma}\in \bR^{r+1}$. If this happens, the pair $\un{n}^\prime,\un{d}^\prime$ cuts out a hyperplane in the space of stability parameters.

When varying the stability parameter, such that the parameter crosses a wall, it is known that one can in principle describe how the moduli spaces changes. Due to the existence of polystable objects, such a description can be difficult to obtain on the level of coarse moduli spaces. In terms of algebraic stacks it is easier to describe how the stack of semistable objects changes when the stability parameter is deformed into a critical value. In other words, we study the behavior of the moduli stack when the stability parameter runs into a wall (and out of it): 
\begin{proposition}\label{walls}
Let $\un{\alpha}$ be a critical stability parameter and $\delta\in \bR^{r+1}$ and fix a rank $\un{n}\in \bN^{r+1}_0$.
\begin{enumerate}
\item There exists $\epsilon>0$ such that for all $0<t<\epsilon$ the stability conditions $\un{\alpha}_t:=\un{\alpha}+t\delta$ coincide for all chains of rank $\un{m}$ if $\un{m} \leq \un{n}$.
\item For any $0<t<\epsilon$ as in (1) we have $$\Chain_{\un{n}}^{\un{d},\un{\alpha}_t-\sst} \subset \Chain_{\un{n}}^{\un{d},\un{\alpha}-\sst}.$$
Moreover, the complement $\Chain_{\un{n}}^{\un{d},\un{\alpha}-\sst}-\Chain_{\un{n}}^{\un{d},\un{\alpha}_t-\sst}$ is the finite union of the $\un{\alpha}_t$--Harder--Narasimhan strata of $\Chain_{\un{n}}^{\un{d}}$ of type $(\un{n}^i,\un{d}^i)$ such that $\mu_{\un{\alpha}}(\un{n}^i,\un{d}^i) =\mu_{\un{\alpha}}(\un{n},\un{d})$ for all $i$.
\end{enumerate}
\end{proposition}

\begin{proof}
To simplify notation, let us abbreviate the $\un{\alpha}_t$ slope by $\mu_t:=\mu_{\un{\alpha}_t}$ and let us write $n:=\sum_{i=0}^r n_i$.
The existence of a constant $\epsilon$ as in (1) is well-known: If $\un{\alpha}_t,\un{\alpha}_s$ define different stability conditions, then there exists chains $\cF_\bullet\subset \cE_\bullet$ and $t\leq x \leq s$ such that $\mu_x(\cF_\bullet)=\mu_x(\cE_\bullet)$ but either $\mu_t(\cF_\bullet)>\mu_t(\cE_\bullet)$ or $\mu_s(\cF_\bullet)> \mu_s(\cE_\bullet)$.

By definition $\mu_t(\cF_\bullet)-\mu_t(\cE_\bullet)= c + t\cdot \frac{\sum m_i \delta_i}{M}$, where $c$ is a rational number, with denominator bounded by $n(n-1)$, $\delta_i\in \bZ$, and $m_i,M$ are integers satisfying $M\leq n(n-1)$ and $m_i\leq n(n-1)$. Therefore $x$ must lie in a discrete subset of $\bR$, which proves the existence of an $\epsilon>0$ as claimed. 

Statement (2) follows from the standard argument for wall-crossing phenomena: If $\cE_\bullet\in \Chain_{\un{n}}^{\un{d}}$ is $\un{\alpha}$-unstable, there exists a subchain $\cF_\bullet \subset \cE_\bullet$ with $\mu_{0}(\cF_\bullet) > \mu_{0}(\cE_\bullet)$, but then we also have $\mu_{t}(\cF_\bullet) > \mu_{t}(\cE_\bullet)$ for small enough $t$.

Similarly, if $\cE_\bullet$ is $\un{\alpha}$-semistable, but $\un{\alpha}_t$ unstable for some $0<t<\epsilon$. Then $\cE_\bullet$ admits a canonical Harder-Narasimhan filtration $\cF_\bullet^1 \subset \dots \subset \cF_\bullet^h=\cE_\bullet$ such that the subquotients $\cF^i_\bullet/\cF^{i-1}_\bullet$ are 
$\un{\alpha}_t$-semistable and $\mu_t(\cF^i_\bullet/\cF^{i-1}_\bullet) > \mu_t(\cF^{i+1}_\bullet/\cF^{i}_\bullet)$. By our assumption on $\epsilon$ the same strict inequalities must then hold for every $t\in (0,\epsilon)$. Since $\cE_\bullet$ is $\un{\alpha}$-semistable, continuity implies that $\mu_0(\cF_\bullet^1)\geq \mu_0(\cF^i_\bullet) \geq \mu_0(\cE_\bullet)$ , so that all subquotients must be $\alpha_0$-semistable of slope $\mu_0(\cE_\bullet)$.

In particular, if $\cE_\bullet^\prime$ is any chain that lies in the same $\un{\alpha}_t$--Harder--Narasimhan stratum as $\cE_\bullet$, then the same numerical conditions hold. Thus, $\cE_\bullet^\prime$  will also be an extension of $\un{\alpha}_0$-semistable chains, so $\cE_\bullet^\prime$ is $\un{\alpha}$--semistable. Thus the whole $\un{\alpha}_t$--Harder--Narasimhan stratum must lie in $\Chain_{\un{n}}^{\un{d},\alpha_0-ss}$.

Since $\Chain_{\un{n}}^{\un{d},\alpha_0-ss}$ is a stack of finite type (\cite{AGS},\cite{STree}), the complement  $$\Chain_{\un{n}}^{\un{d},\alpha_0-\sst}-\Chain_{\un{n}}^{\un{d},\alpha_t-\sst}$$ can only contain finitely many non-empty Harder-Narasimhan strata.
\end{proof}

\begin{remark}
In the next section we will find a result (Corollary \ref{finite}) which allows to bound the number of
Harder--Narasimhan strata occuring in the complement $\Chain_{\un{n}}^{\un{d},\alpha_0-\sst}-\Chain_{\un{n}}^{\un{d},\alpha_t-\sst}$ effectively.
\end{remark}

We will apply the above proposition as follows: By \cite[Proposition 4.8]{HGS} we know that for stability parameters satisfying ($\star$) the class of a Harder-Narasimhan stratum in $\K0hat$
can be expressed as a product of clases of stacks of semistable chains of smaller rank times a multiple of $\bL$. Therefore, to compute the class of $\Chain_{\un{n}}^{\un{d},\un{\alpha}-\sst}$ for any $\alpha$ satisfying ($\star$), it would be sufficient to know on the one hand the class for large stability parameters and to compute inductively classes of spaces of semistable chains
of smaller rank. 


\section{Necessary conditions for the existence of semistable chains}\label{sec:conditions}
In this section we want to collect conditions on $\un{n},\un{d},\un{\alpha}$ which are necessary for the existence of $\un{\alpha}$-semistable chains of rank $\un{n}$ and degree $\un{d}$. Some of these conditions are well known, i.e., conditions (1) appears in \cite[Proposition 2.4]{AGS},  but we could not find the other conditions in the literature. 

\begin{proposition}\label{conditions}
Fix $\un{n}=(n_0,\dots,n_r)\in \bN^{r+1}$, $\un{d}=(d_0,\dots,d_r)\in \bZ^{r+1}$ and $\un{\alpha}=(\alpha_r,\dots,\alpha_0)\in\bR^{r+1}$ a stability parameter satisfying $\alpha_r>\alpha_{r-1}>\dots>\alpha_0$. 
Write $$\mu_\bullet:=\frac{\sum_{i=0}^r (d_i + \alpha_i n_i)}{\sum_{i=0}^r{n_i}}.$$
An $\alpha$-semistable chain $\cE_\bullet$ of rank $\un{n}$ and degree $\un{d}$ can only exist if
\begin{enumerate}
\item For all $j\in \{0,\dots,r-1\}$ we have 
$$ \frac{\sum_{i=0}^j \big(d_i + \alpha_i n_i\big)}{\sum_{i=0}^j n_i} \leq \mu_\bullet$$ 
\item For all $j$ such that $n_j=n_{j-1}$ we have $$\frac{d_j}{n_j}\leq \frac{d_{j-1}}{n_{j-1}}.$$ 
\item For all $0\leq k < j \leq r$ such that $n_j <\min\{n_{k},\dots,n_{j-1}\}$ we have:
$$\displaystyle{\frac{\sum_{i\not\in[k,j]}\big(d_i+\alpha_i n_i \big)+(j-k+1)d_j+\big(\sum_{i=k}^j\alpha_i\big)n_j}{\sum_{i\not\in[k,j]}n_i+(j-k+1)n_j}\leq \mu_\bullet}, \textrm{ i.e.,}$$
$$ \mu_{\un{\alpha}}(\cE_r \to \dots \to \cE_{j+1} \to \cE_j \to \dots \to \cE_j \to \cE_{k-1} \to \dots \to \cE_0) \leq \mu_\bullet.$$
\item For all $0\leq k < j \leq r$ such that $n_j > \max\{n_{k},\dots,n_{j-1}\}$ we have:
$$\frac{\sum_{i=k+1}^{j} \big( d_i-d_k +\alpha_i(n_i-n_k) \big)}{\sum_{i=k+1}^j(n_i-n_k)} \leq \mu_\bullet, \textrm{ i.e.,}$$ 
$$\mu_\bullet \leq \mu_{\un{\alpha}}(\cE_r \to \dots \to \cE_{j+1} \to \cE_k \to \dots \to \cE_k \to \cE_{k-1} \to \dots \to \cE_0).$$
\end{enumerate}
\end{proposition}

\begin{proof}
The inequalities (1) simply state that every chain contains the obvious subchains $0\to \dots \to 0 \to \cE_j \to \cE_{j-1} \to \dots \to \cE_0$, so that the slope of this subchain has to be $\leq \mu_\bullet$.

To see (2) suppose that $n_j=n_{j-1}$ and $d_j>d_{j-1}$. If this happens the map $\cE_j\to\cE_{j-1}$ cannot be injective,
so that we find a subchain
\begin{center}
\begin{tabular}{rp{5em}cp{5em}l}
$\cdots\to$& $\ker(\phi_j)$ & $\to$ & $ 0 $&$\to \cdots$\\
\end{tabular}
\end{center}
and a quotient
\begin{center}
\begin{tabular}{rp{5em}cp{5em}l}
$\cdots\to$ & $ 0$ &$ \to$ & $\coker(\phi_{j})$ & $\to \cdots.$\\
\end{tabular}
\end{center}
 We know that $\rk(\ker(\phi_j))=\rk(\coker(\phi_j))$ and $\deg(\ker(\phi_j)) > \deg(\coker(\phi_j))$. However, if $\cE_\bullet$ was semi-stable we would have 
$$\mu_\bullet \geq \mu_{\un{\alpha}}(\ker(\phi_j)) = \mu(\ker(\phi_j))+\alpha_j > \mu(\coker(\phi_j)) + \alpha_{j-1} = \mu_{\un{\alpha}}(\coker(\phi_j)) \geq \mu_\bullet,$$
which cannot happen.

Condition (3) will follow by considering the subchain 
$$\cE_r \to \dots \cE_j \to \im(\phi_j) \to \im(\phi_{j-1}\circ\phi_j) \dots \to \cE_{k-1} \to \dots \to \cE_0.$$
To clarify the argument, let us first assume that $k=j-1$. We have subchains:
$$\xymatrix{
\cK_\bullet:=&\cdots\ar[r] & 0\ar[r]\ar[d]  &\ker(\phi_j) \ar[r]\ar[d] & 0 \ar[r]\ar[d]  & 0 \ar[r]\ar[d] &\cdots\\
\cE^\prime_\bullet:=&\cdots\ar[r] &\cE_{j+1}\ar[r]\ar[d]  & \cE_j \ar[r]\ar[d]  & \im(\phi_j)\ar[r]\ar[d]  & \cE_{j-2}\ar[d]\ar[r] &\cdots  \\
\cE_\bullet= &\cdots\ar[r] &\cE_{j+1}\ar[r]  & \cE_j \ar[r]  & \cE_{j-1}\ar[r]  & \cE_{j-2}\ar[r] &\cdots
}$$
If $\ker(\phi_j)=0$ then the slope of the subchain defined using  $\im(\cE_j)$ is the expression given in (3).
Otherwise, the expression equals 
$$\mu_{\un{\alpha}}(\cE_\bullet^\prime\oplus \cK_{\bullet-1})= \frac{\rk(\cE_\bullet^\prime)}{\rk(\cE_\bullet^\prime)+\rk(\cK_\bullet)} \mu_{\un{\alpha}}(\cE_\bullet^\prime) 
+ \frac{\rk(\cK_\bullet)}{\rk(\cE_\bullet^\prime)+\rk(\cK_\bullet)} (\mu_{\un{\alpha}}(\cK_\bullet)-(\alpha_j-\alpha_{j-1}).$$
If the left hand side of this expression is $>\mu_\bullet$ then we find that either $\mu_{\un{\alpha}}(\cE_\bullet^\prime)>\mu_\bullet$ or $\mu_{\un{\alpha}}(\cK_\bullet)>\mu_\bullet+(\alpha_j-\alpha_{j-1})>\mu_\bullet$. This proves the necessity of the inequality for $k=j-1$.

In general fix $k$ and denote by $\LHS$ the left hand side of the inequality in (3) and suppose that $\LHS>\mu_\bullet$.  Let us abbreviate $\phi_{k,j}:=\phi_{k}\circ \dots \circ \phi_{j-1} \circ \phi_j$.

We have a subchain 
$$\cE_\bullet^\prime:= \cdots \cE_{j+1} \to \cE_j \to \im(\phi_j) \to \im(\phi_{j-1,j}) \to \dots \to \im(\phi_{k+1,j}) \to \cE_{k-1} \to \dots.$$
As before, if $\ker(\phi_{k+1,j})=0$ then $\cE^\prime_\bullet$ is the chain
$$ \dots \cE_{j+1} \to \cE_j \map{\id} \cE_j \map{\id} \dots \map{\id} \cE_j \map{\phi_{k,j}} \cE_{k-1} \to \dots $$
so that condition (3) is given by $\LHS = \mu_{\un{\alpha}}(\cE_\bullet^\prime)\leq \mu_\bullet$.

If the kernels are non-zero then we have 
$$\LHS = \mu_{\un{\alpha}}(\to \cE_j \to \im(\phi_j)\oplus \ker(\phi_j) \to \im(\phi_{j-1,j}) \oplus \ker(\phi_{j-1,j}) \to \dots ).$$
Since the slope of a sum of chains is a convex combination of the summands either $\mu_{\un{\alpha}}(\cE_\bullet^\prime)\geq\LHS$ or there exists an $i_0$ with $j\leq i_0<k$ such that $\mu(\ker(\phi_{i_0,j}))+\alpha_{i_0}>\LHS$ and $\mu(\ker(\phi_{l,j}))+\alpha_l\leq \LHS$ for all $l>i_0$.
If the latter happens, consider the exact sequence 
$$ 0 \to \ker(\phi_{i_0+1,j}) \to \ker(\phi_{i_0,j}) \to \phi_{i_0+1,j}(\ker(\phi_{i_0,j})) \to 0.$$
We know that $\mu(\ker(\phi_{i_0,j}))+\alpha_{i_0+1} > \mu(\ker(\phi_{i_0,j}))+\alpha_{i_0}>\LHS$ and $\mu(\ker(\phi_{i_0+1,j}))+\alpha_{i_0+1}<\LHS$.
Therefore $\mu(\phi_{i_0+1,j}(\ker(\phi_{i_0,j})))+\alpha_{i_0+1} > \LHS > \mu_\bullet$, so that the subchain 
$$0 \to \phi_{i_0+1,j}(\ker(\phi_{i_0,j})) \to 0$$
 is a destabilizing subchain of $\cE_\bullet$.
This proves that condition (3) is necessary.

The last condition (4) follows from (3) by passing to the dual chain.
\end{proof}

\begin{corollary}\label{finite}
For fixed values of $\un{n}\in \bN^{r+1}$ and $\un{\alpha}\in\bR^{r+1}$ satisfying $\alpha_r<\alpha_{r-1}<\dots < \alpha_0$ and $d\in \bZ$ there are only finitely many values $\un{d} \in \bZ$ with $\sum d_i=d$ which satisfy the conditions given in Proposition \ref{conditions}. 

In particular, there exist only finitely many $\un{d}\in\bZ^{r+1}$ with $\sum d_i=d$ such that $\Chain_{\un{n}}^{\un{d},\un{\alpha}-\sst}\neq \emptyset$.
\end{corollary}
\begin{proof}
The conditions given in Proposition \ref{conditions} are linear inequalities in $\un{d}$. Thus we only need to check that these cut out a compact polytope in the hyperplane $\sum d_i = d$.  For $i=1,\dots r$ let us denote by $w_i$ the vector $(\underbrace{1,\dots,1}_{i-\text{times}}, 0,\dots,0)\in \bR^{r+1}$, and by $\langle,\rangle$ the standard inner product on $\bR^{r+1}$, so that  condition (1) can be written as 
$$ \langle \un{d}, w_i \rangle \geq c_i$$
where $c_i$ is an explicit constant depending only on $i,\un{n},\un{\alpha}$ and $d=\sum d_i$.
Also for $i=1,\dots r$ if $n_i=n_{i-1}$ we put $v_i:=(0,\dots,0,-1,1,0,\dots,0)$, where the $(i-1)$-th coordinate is $-1$ and the $i$-th coordinate is $1$.
if $n_i<n_{i-1}$ we put $v_i:=(1,\dots,1,0,2,1,\dots,1)$ where the $(i-1)$-th coordinate is $0$ and the $i$-th coordinate is $2$ and if $n_i>n_{i-1}$ we put
$v_i:=(-1,\dots,-1,-2,0,-1,\dots,-1)$ here the $(i-1)$-th coordinate is $-2$ and the $i$-th coordinate is $0$.

With this definition the conditions (2),(3) and (4) are of the form
$$ \langle \un{d}, \un{r}_i \rangle \geq c^\prime_i$$
where again $c_i^\prime$ is an explicit constant depending only on $i,\un{n},\un{\alpha}$ and $d=\sum d_i$.

To show that the these conditions cut out a compact polytope in the hyperplane $\sum d_i = d$ we only need to check that the projections $\overline{v}_i,\overline{w}_i$ of $v_i,w_i$ onto the hyperplane $\sum d_i =0$ satisfy a linear relation $\sum a_i v_i +\sum b_i w_i =0$ with non negative coefficients, such that the vectors occurring with non zero coefficients span the hyperplane.

However the projection of $v_i$ is $\overline{v}_i=(0,\dots,0,-1,1,0,\dots,0)$ for all $i$ so these vectors are a basis for the hyperplane $\sum d_i=0$ and
$\overline{w}_1=\frac{1}{r+1}(r,-1,\dots,-1)$. Thus $\sum (r+1-i) \overline{v}_i + (r+1) \overline{w}_1 =0$, which proves our claim.
\end{proof}

\begin{remark}
It would be interesting to see whether these conditions might be sufficient conditions for the existence of semistable chains for parameters $\alpha$ satisfying ($\star$).
\end{remark}

\section{Applications to classes of moduli spaces of chains and Higgs bundles}\label{sec:application}

Next we want to apply our necessary conditions to find --- for any rank $\un{n}$ and degree $\un{d}$ and any parameter $\un{\alpha}$ satisfying ($\star$) --- a family of stability conditions $\un{\alpha}_t= \un{\alpha} +t\delta$, such that for all $t\geq 0$ the parameter $\un{\alpha}_t$ satisfies ($\star$) and such that we can compute $\Chain_{\un{n}}^{\un{d},\alpha_t-\sst}$ for $t\gg 0$. 

Let us first assume that $n_i\neq n_j$ for some $i,j$. Then we can even find such a family $\un{\alpha}_t$ such that the moduli space is empty for $t\gg 0$:
\begin{lemma}\label{goodalpha}
Fix $\un{n}\in \bN^{r+1},\un{d}\in \bZ^{r+1}$ together with a stability parameter $\alpha\in \bR^{r+1}$ satisfying ($\star$). Suppose $n_r\neq n_i$ for some $0\leq i<r$.
\begin{enumerate}
\item If $n_r=n_{r-1}=\dots =n_{k+1} < n_{k}$ for some $k$, define $\delta=(\delta_i)$ as $\delta_i:=1$ for $i> k$ and $\delta_i:=0$ for $i\leq k$.
\item If $n_r=n_{r-1}=\dots=n_{k+1} > n_{k}$ define $\delta=(\delta_i)$ as $\delta_i:=0$ for $i>k$ and $\delta_i:=-1$ for $i\leq k$.
\end{enumerate}
Let $\alpha_t:=\alpha+t\delta$. Then $\alpha_t$ satisfies ($\star$) for all $t\geq 0$ and $\Chain_{\un{n}}^{\un{d},\alpha_t-\sst}=\emptyset$ for $t \gg 0$.
\end{lemma}
\begin{proof}
The proof is an easy application of the necessary conditions found in Proposition \ref{conditions}: In both cases $\un{\alpha}_t$ satisfies ($\star$) for $t\geq 0$, because $\un{\alpha}$ was assumed to satisfy ($\star$) and $\delta_i\geq \delta_{i-1}$ for all $i$.

In case (2) we  defined $\delta$ to have negative coefficients, so that 
$$\lim_{t\to \infty} \mu_t(\un{n},\un{d}) = \lim_{t\to \infty} \frac{\sum_{i=0}^r (d_i + (\alpha_i+t\delta_i) n_i)}{\sum_{i=0}^r{n_i}} = -\infty.$$ On the other  condition (4) of Proposition \ref{conditions} applied to $j=k+1$ gives
$$ \frac{d_{k+1}-d_k +(\alpha_{k+1}+t\delta_{k+1})(n_{k+1}-n_k)}{n_{k+1}-n_k)} \leq \mu_{t}(\un{n},\un{d})$$
where the left hand side is independent of $t$ because $\delta_{k+1}=0$. Thus for $t\gg 0$ this condition cannot be satisfied.

In case (1) the dual argument gives the result: Here
$$\lim_{t\to \infty} \mu_t(\un{n},\un{d}) = \lim_{t\to \infty} \frac{\sum_{i=0}^r (d_i + (\alpha_i+t\delta_i) n_i)}{\sum_{i=0}^r{n_i}} = \infty.$$ 
And for $j=k+1$ condition (3) of Proposition \ref{conditions} is equivalent to 
$$\mu_t(\un{n},\un{d})) \leq \frac{d_k-d_{k+1}}{n_k-n_{k+1}}+\alpha_{k}+t\delta_k.$$ 
Since $\delta_k=0$ the right hand side is independent of $t$, so that the condition cannot hold for $t\gg 0$.  
\end{proof}

If $\un{n}=(n,n,\dots,n)$ we already gave an example for large stability parameters in \cite[Corollary 6.10]{HGS}, which we can use:
\begin{lemma}\label{recursnn}
Let $\un{n}=(n,\dots,n)$ and let $\alpha$ satisfy ($\star$). Let $\delta_i:=i$ and put $\un{\alpha}_t:=\un{\alpha} + t\un{\delta}$. 

Then the parameter $\alpha_t$ satisfies ($\star$) for all $t\geq 0$ and for $t\gg 0$ the parameter $\alpha_t$ satisfies $ \alpha_i-\alpha_{i-1} > d_i-d_{i-1}.$
For such $t$ we have:
\begin{align*}
[\Chain_{\un{n}}^{\un{d},\alpha_t-\sst}] &= [\Bun_n^{d_0}] \prod_{i=1}^r [\Sym^{d_{i-1}-d_i}(C\times \bP^{n-1})]\\
                       &-\bigg(\sum_{\un{m},\un{e},\un{k}} \bL^{\sum_{k<j} \chi_{kj}} \prod_i [\Chain_{\un{m}_j}^{\un{e_j},\alpha_t-\sst}]\bigg),
\end{align*}
where the sum runs over all partitions $n=\sum_{j=1}^l m_j$, $\un{d}=\sum_j \un{e}^{(j)}$ such that for all $i,j$ we have $e_i^{(j)} \leq e_{i-1}^{(j)}$ and for $\mu^{(j)}:=\frac{\sum_i e_i^{(j)}}{rm_j}$ we have $\mu_1 > \dots > \mu_l$.
We have written $\chi_{kj}= m_jm_k(g-1)+ \sum_{i=0}^r (m_ke^{(j)}_i-m_je^{(k)}_i)- \sum_{i=1}^r (m_ke^{j}_i -m_je^{(k)}_{i-1})$.
\end{lemma}
\begin{proof}
This is a reformulation \cite[Corollary 6.10]{HGS}.
\end{proof}

We can now use these observations to apply the wall-crossing procedure:
\begin{theorem}\label{factorP1}
Let $\un{n}\in\bN^{r+1},\un{d}\in \bZ^{r+1}$ and assume that $\un{\alpha}\in \bR^{r+1}$ satisfies ($\star$), i.e. $\alpha_{i}-\alpha_{i-1}\geq 2g-2$. Then the following hold:
\begin{enumerate}
\item $[\Chain_{\un{n}}^{\un{d},\alpha-\sst}]$ can be expressed in terms of $\bL$ and the symmetric powers $[C^{(i)}]$ of the curve $C$.
\item $[\Chain_{\un{n}}^{\un{d},\alpha-\sst}]$ is divisible by $[\Bun_1^0]$, i.e. there exists a class $[\PChain_{\un{n}}^{\un{d},\alpha-\sst}]\in \K0hat$ such that $$[\Chain_{\un{n}}^{\un{d},\alpha-\sst}]= [\Bun_1^0]\cdot [\PChain_{\un{n}}^{\un{d},\alpha-\sst}].$$
\item If $n_r\neq n_i$ for some $i$, then the class $[\Chain_{\un{n}}^{\un{d},\alpha-\sst}]$ is divisible by $[\Bun_1^0]^2$, i.e. there exists a class $[\PPChain_{\un{n}}^{\un{d},\alpha-\sst}]\in \K0hat$ such that $$[\Chain_{\un{n}}^{\un{d},\alpha-\sst}]= [\Bun_1^0]^2\cdot [\PPChain_{\un{n}}^{\un{d},\alpha-\sst}].$$
\end{enumerate}
\end{theorem}
\begin{proof}
We only have to combine Proposition \ref{walls}, which describes how the stack of semistable chains changes if one varies the stability parameter and the fact that we can describe $\Chain_{\un{n}}^{\un{d},\un{\alpha-}\sst}$ for large stability parameters (Lemma \ref{goodalpha} and Remark \ref{recursnn}): For any $\un{n},\un{d}$ we can by Lemma \ref{goodalpha} and Lemma \ref{recursnn} find a family $\un{\alpha}_t = \un{\alpha} +t \delta$ of stability parameters satisfying ($\star$), such that for $t \gg 0$ the class of $[\Chain_{\un{n}}^{\un{d},\alpha_t-\sst}]$ is either $0$ or computed by the recursion given in \cite[Corollary 6.10]{HGS}. Since all of the terms appearing in the recursive formula of the expressions given in \cite[Corollary 6.10]{HGS} satisfy the statements (1) and (2) we see that for $t\gg 0$ the classes $[\Chain_{\un{n}}^{\un{d},\alpha_t-\sst}]$ satisfy the claims (1),(2) and (3). 

Proposition \ref{walls} implies that the difference $[\Chain_{\un{n}}^{\un{d},\un{\alpha}_t-\sst}]-[\Chain_{\un{n}}^{\un{d},\un{\alpha}-\sst}]$ is given by an alternating sum of classes of finitely many Harder--Narasimhan strata for some stability parameters $\un{\alpha}_s$ with $0\leq s \leq t$. Since $\un{\alpha}_s$ satisfies  $\alpha_{s,i}-\alpha_{s,i-1}\geq 2g-2$ for all $i$ we can apply \cite[Proposition 4.8]{HGS} to see that the classes of the Harder--Narasimhan strata are given by products $$\bL^\chi \prod_{i=1}^h [\Chain_{\un{n}_i}^{\un{d}_i,\un{\alpha}_s-\sst}],$$
where $\un{n}_i < \un{n}$. By induction on the rank $\un{n}$ we already know that all of these products are divisible by $[\Bun_1^0]^h$ with $h\geq 2$. This implies that the claims (1),(2) and (3) also hold for $[\Chain_{\un{n}}^{\un{d},\un{\alpha}-\sst}]$.
\end{proof}

\begin{remark}\label{algorithm}
The above proof is constructive. In particular, we can apply this to obtain an algorithm to compute the class of the moduli space of semi-stable Higgs bundles of coprime rank and degree for any curve and fixed numerical invariants:  As we recalled in Section \ref{sec:preliminaries} the class of the moduli stack of is can be computed as:
$$ [\cM_{n}^d] = \bL^{n^2(g-1)}\sum_{\un{n},\un{d} \in I} [\Chain_{\un{n}}^{\un{d},\alpha-\sst}],$$
where $\un{\alpha}=(\alpha_i)$ with $\alpha_i=i(2g-2)$.

For any partition $\un{n}$ of $n$ Corollary \ref{finite} gives a finite bound for the degrees $\un{d}$ that give a non zero contribution in the above summation, so that in principle we are reduced resulting in a finite computation once we fixed the fixed rank $n$ of our bundles and the genus $g$ of our curve.
\end{remark} 

\begin{corollary}
Suppose $n\in \bN$ and $d\in\bZ$ are coprime, then the class of the moduli
space of semistable Higgs bundles $[M_n^d]\in \K0hat$ can be expressed in
terms of $\bL$ and the symmetric powers $[C^{(i)}]$ of the curve $C$.
\end{corollary}
\begin{proof}
Since we know that $[M_n^d]$ can be expressed in terms of the classes of moduli spaces of semistable chains \cite[Corollary 2.2]{HGS} this follows immediately from the Theorem \ref{factorP1}.
\end{proof}

\section{Computation of the $y$-genus}\label{main}

In this section we apply Theorem \ref{factorP1} to a conjecture of Hausel, predicting the $y$-genus of the moduli space of $\PGL_n$-Higgs bundles \cite[Conjecture 5.7]{HauselMirror},  which is a specialization of a more general conjecture of Hausel and Rodriguez-Villegas \cite[Conjecture 5.6]{HauselMirror}. Thus we will now assume that our base field $k=\bC$. 

Let us recall that the class of the moduli stack of line bundles on our curve $C$ is given by $[\Bun_1^0]= \frac{P(1)}{\bL-1}.$
For the corresponding $E$-polynomials we have $E(P(1),x,y)=(1-x)^g(1-y)^g$, $E(\bL,x,y)=xy$, so that in particular we find 
$$H_y([\Bun_1^0],y)=0.$$
In particular Theorem \ref{factorP1} implies that $H_y(\Chain_{\un{n}}^{\un{d},\alpha-\sst},y)=0$ for all stability parameters $\alpha$ satisfying ($\star$).

To obtain interesting invariants Hausel suggested to consider $\PGL_n$-Higgs bundles instead of $\GL_n$-Higgs bundles.

Let us denote by $\cM^d(\PGL_n)$ the moduli stack of $\PGL_n$-Higgs bundles of degree $d\in \bZ/n\bZ$ on $C$, i.e., this is the stack classifying pairs $(\cP,\theta)$, where $\cP$ is a principal $\PGL_n$-bundle of degree $d$ and $\theta \in H^0(C, \cP \times^{\PGL_n} \Lie(\PGL_n))$.

The natural map $\GL_n \to \PGL_n$ induces a surjective map $\cM_n^d \to \cM^d(\PGL_n)$ and this makes $\cM_n^d$ into a torsor under $\cM_1^0$ over $\cM^d(\PGL_n)$.  To state this in a more down to earth language, let us denote by $\cM^{d,\sst}(\PGL_n)\subset \cM^d(\PGL_n)$ the substack of semistable $\PGL_n$ bundles, which can be defined as the image of $\cM_n^{d,\sst}$ in $\cM^d(\PGL_n)$.  The group $\Pic^0(C)$ acts on $M_n^d$ and in case $(n,d)=1$ it is not hard to see that $\cM^{d,\sst}(\PGL_n)\cong [M_n^d/T^*\Pic^0(C)]$. 

This description makes it easy to compare the cohomology groups of these spaces. Namely, since $T^*\Pic^0(C)\cong H^0(C,\Omega_C) \times \Pic^0$ we can apply (\cite[Proposition 7.4.5]{NgoLemme}) to $[(M_n^d/H^0(C,\Omega_C))/\Pic^0(C)]$ to obtain $$H^*_c(M_n^d,\bQ)\cong H^*(\Pic^0(C),\bQ)\tensor H^*_c(\cM^d(\PGL_n))\tensor H^*_c(\bA^g,\bQ).$$

Now we can explain how Theorem \ref{factorP1} allows us to compute $H_y(\cM^{d,\sst}(\PGL_n),y)$:  

We have already seen (Theorem \ref{factorP1} and  \cite[Corollary 2.2]{HGS}) that the class $[M_n^d]$ is divisible by $P(1)$, i.e.,
$[M_n^d]= P(1) [PM_n^d]$ for an explicitly computable class $[PM_n^d]\in \K0hat$. Therefore, we see that 
$$H_y(\cM^{d,\sst}(\PGL_n),y)=H_y([PM_n^d],y) (y)^{-g}.$$
To make this explicit let us once more denote by $$I:= \{ (\un{n},\un{d}) | \sum_i n_i =n, \sum_i d_i+i(2g-2)n_i=d \}$$ and fix $\un{\alpha}=(\alpha_i)$ with $\alpha_i=i(2g-2)$. Then we have
$$ H_y(\cM^{d,\sst}(\PGL_n),y) = (y)^{n^2(g-1)-g} \left(\sum_{(\un{n},\un{d})\in I}  H_y([\PChain_{\un{n}}^{\un{d},\alpha-\sst}],y\right).$$

We have seen that $H_y(P(1),y)=0$. Moreover for $\un{n}$ with $n_i\neq n_j$ we have $[\Chain_{\un{n}}^{\un{d},\alpha-\sst}]=P(1)^2 [\PPChain]$. Therefore, in the above summation, only summands with $\un{n}=(m,\dots,m)$ contribute to the $y$-genus of $\cM^d(\PGL_n)$.
By the wall-crossing result Proposition \ref{walls} we know moreover, that the $y$-genus of these terms does not depend on the stability parameter $\un{\alpha}$, as long as $\alpha$ is chosen to satisfy ($\star$), i.e.\ for $\alpha_{t}$ chosen as in Lemma \ref{goodalpha} and Lemma \ref{recursnn} we have
$$   H_y([\PChain_{\un{n}}^{\un{d},\un{\alpha}-\sst}],y) =  H_y([\PChain_{\un{n}}^{\un{d},\un{\alpha}_t-\sst}],y) \text{ for all } t\geq 0.$$
Thus the recursive formula \cite[Corollary 6.10]{HGS} shows that for $\un{n}=(m,\dots,m)$ we have:
\begin{align*}
H_y([\PChain_{\un{n}}^{\un{d},\un{\alpha}-\sst}],y)=& \prod_{i=2}^{m-1} (1-y^{i})^{g-1}(1-y^{i+1})^{g-1} \\
                                              &\cdot \prod_{i=1}^{n/m-1} H_y(\Sym^{d_i-d_{i+1}} (C\times \bP^{m-1}),y).
\end{align*}
Note that we have
$$\sum_{k=0}^\infty [\Sym^k(C\times \bP^{m-1}]t^k = Z(C\times \bP^{m-1},t)= \prod_{i=0}^{m-1} Z(C,\bL^it),$$
so that by Example \ref{example} we have $H_y(\Sym^{k} (C\times \bP^{m-1}),y)=0$ for $k>m(2g-2)$.

Thus, summing over all $\un{n},\un{d}$ we find:
\begin{align*}
H_y(\cM^{d,\sst}(\PGL_n),y) &= y^N \sum_{m|n} \sum_{d_0\leq d_1 \leq \dots \leq d_{n/m-1} \atop \sum d_i +i(2g-2)m  = d} H_y([\PChain_{\un{n}}^{\un{d},\alpha-\sst}],y) \\
&=y^N  \sum_{m|n} \Big( \prod_{i=2}^{m-1} (1-y^{i})^{g-1}(1-y^{i+1})^{g-1} \cdot\\
& \sum_{k_1,\dots,k_{n/m-1}\geq 0 \atop \sum ik_i \equiv d \mod n/m}  \prod_{i=1}^{n/m-1} H_y(\Sym^{k_i}(C\times \bP^{m-1}),y) \Big). \\
\end{align*}

Let us fix $m|n$ and let $\zeta:= e^{2\pi i \frac{m}{n}}$ be a primitive $\frac{n}{m}$-th root of unity. Then:
\begin{align*}
& \sum_{k_1,\dots,k_{\frac{n}{m}-1}\geq 0 \atop \sum ik_i \equiv d \mod \frac{n}{m}}  \prod_{i=1}^{\frac{n}{m}-1} H_y(\Sym^{k_i}(C\times \bP^{m-1}),y) \\
&= \frac{m}{n}\sum_{l=1}^{\frac{n}{m}} \zeta^{-ld} \sum_{k_1,\dots,k_{\frac{n}{m}-1}\geq 0}  \prod_{i=1}^{\frac{n}{m}-1} \zeta^{lik_i} H_y(\Sym^{k_i}(C\times \bP^{m-1}),y)\\
&= \frac{m}{n}\sum_{l=1}^{\frac{n}{m}} \zeta^{-ld} \prod_{i=1}^{\frac{n}{m}-1} H_y(Z(C\times \bP^{m-1},\zeta^{li}),y)\\
&= \frac{m}{n}\sum_{l=1}^{\frac{n}{m}} \zeta^{-ld} \prod_{j=0}^{m-1}\prod_{i=1}^{\frac{n}{m}-1} H_y(Z(C\times \bA^{j},\zeta^{li}),y)\\
&= \frac{m}{n}\sum_{l=1}^{\frac{n}{m}} \zeta^{-ld} \prod_{j=0}^{m-1}\prod_{i=1}^{\frac{n}{m}-1} (1-y^j\zeta^{li})^{g-1}(1-y^{j+1}\zeta^{li})^{g-1}
\end{align*}
since one of the factors for $j=0$ vanishes if $l$ is not coprime to $\frac{n}{m}$ this equals
\begin{align*}
&= \frac{m}{n}\sum_{1\leq l \leq \frac{n}{m} \atop (l,\frac{n}{m})=1} \zeta^{-ld} \left(\frac{n}{m}\right)^{g-1} \left(\prod_{j=1}^{m-1} \left(\frac{1-y^{j\frac{n}{m}}}{1-y^j}\right)^{2}\right)^{g-1} \left(\frac{1-y^n}{1-y^m}\right)^{g-1}\\
&= \frac{m}{n} \mu\left(\frac{n}{m}\right) \left(\frac{n}{m}\right)^{g-1} \prod_{j=1}^{m-1} \left(\frac{(1-y^{j\frac{n}{m}})^2}{(1-y^j)^2}\right)^{g-1} \left(\frac{1-y^n}{1-y^m}\right)^{g-1},
\end{align*}
because $\displaystyle{\mu(k)= \sum_{1\leq j \leq k \atop (j,k)=1} e^{\frac{2\pi i}{k}j}.}$
Summing over all $m|n$ and replacing $m$ by $\frac{n}{m}$ we find:
\begin{align*}
H_y(\cM^{d,\sst}(\PGL_n),y) &= y^N\sum_{m|n} \Big( \prod_{i=2}^{\frac{n}{m}-1} (1-y^{i})^{g-1}(1-y^{i+1})^{g-1} \cdot\\
                     & \frac{\mu(m)}{m} m^{g-1} \prod_{j=1}^{\frac{n}{m}-1} \left(\frac{(1-y^{jm})^2}{(1-y^j)^2}\right)^{g-1} \left(\frac{1-y^n}{1-y^{\frac{n}{m}}}\right)^{g-1}\Big)\\
                     &= y^N \left(\frac{1-y^n}{1-y}\right)^{g-1} \sum_{m|n}  \frac{\mu(m)}{m} \Big(m\prod_{j=1}^{\frac{n}{m}-1}(1-y^{jm})^2   \Big)^{g-1}.
\end{align*}
\begin{theorem}[Hausel's conjecture for the $y$-genus]
Let $n\in \bN$ and $d\in \bZ$ be coprime. Then the  $y$-genus of $\Higgs_n^{d,\sst}$ is given by:
\begin{align*}
H_y(\cM^{d,\sst}(\PGL_n),y)=& y^N \left(\frac{1-y^n}{1-y}\right)^{g-1} \sum_{m|n}  \frac{\mu(m)}{m} \left(m \prod_{j=1}^{\frac{n}{m}-1}(1-y^{jm})^2   \right)^{g-1}.
\end{align*}
\end{theorem}
\begin{remark}
Because our computation is made using the Grothendieck ring of varieties, we use the $y$-genus with compact supports in the above theorem.
In \cite{HauselMirror} Hausel formulates his conjecture for the $y$-genus without supports, which can be read off from the above result as follows:
Since the cohomology of $\cM^{d,\sst}(\PGL_n)$ is pure, Poincar\'e duality allows to deduce the Hodge polynomial 
$$H(\cM^{d,\sst}(\PGL_n);x,y):= \sum_{p,q} \dim H^{p,q}(\cM^{d,\sst}(\PGL_n)) x^py^q $$ 
from the E-polynomial of $\cM^{d,\sst}(\PGL_n)$ as:
$$H(\cM^{d,\sst}(\PGL_n);-x,-y)= (xy)^{2N} E(\cM^{d,\sst}(\PGL_n);x^{-1},y^{-1}).$$
The $y$-genus without supports being defined to be $H(\cM^{d,\sst}(\PGL_n);-1,y)$ we find that this equals $(-y)^{2N}H_y(\cM^{d,\sst}(\PGL_n),-y^{-1})$,
which gives the formula conjectured by Hausel.
\end{remark}

\end{document}